\documentclass[11pt]{amsart}

\usepackage[margin=1.4in]{geometry}

\usepackage{amsmath,amssymb,amsthm}

\usepackage{hyperref}

\theoremstyle{plain}
\newtheorem{theorem}{Theorem}[section]
\newtheorem{lemma}[theorem]{Lemma}

\theoremstyle{definition}
\newtheorem{definition}[theorem]{Definition}

\theoremstyle{remark}

\newcommand{\C}{\mathbb{C}}
\newcommand{\R}{\mathbb{R}}

\renewcommand{\S}{\mathbb{S}}

\renewcommand{\b}{\mathbf{b}}
\newcommand{\n}{\mathbf{n}}

\newcommand{\Pcal}{\mathcal{P}}
\newcommand{\Qcal}{\mathcal{Q}}
\newcommand{\Ccal}{\mathcal{C}}

\newcommand{\Tsf}{\mathsf{T}}
\newcommand{\Nsf}{\mathsf{N}}
\newcommand{\psf}{\mathsf{p}}

\DeclareMathOperator{\Sym}{Sym}
\DeclareMathOperator{\tr}{tr}

\numberwithin{equation}{section}

\title[Loss of normal flatness under mean curvature flow]{Counterexamples to preservation of flat normal bundles under mean curvature flow}
\author[K. Kunikawa]{Keita Kunikawa}
\address{Department of Mathematical Sciences, Tokushima University, 2-1 Minamijyousanjima-cho, Tokushima 770-8506, Japan}
\email{kunikawa@tokushima-u.ac.jp}
\date{}

\begin{document}

\begin{abstract}
We construct an embedded torus and an entire graph in $\R^4$ whose normal bundles are initially flat but lose this property instantaneously under mean curvature flow. We also give an example showing that the parallel principal normal condition is not preserved under mean curvature flow, even though the normal bundle remains flat along the flow.
\end{abstract}

\maketitle

\section{Introduction}
\subsection{Flat normal bundles and mean curvature flow}
Mean curvature flow in higher codimension differs substantially from the hypersurface case. For hypersurfaces, the normal bundle is one-dimensional and there is a single shape operator, while in higher codimension the normal connection may have nonzero curvature and several shape operators interact with one another. These additional geometric quantities make the evolution equations more complicated and prevent many arguments from the hypersurface theory from extending directly to arbitrary codimension. See, for example, \cite{Smoczyk12} for an overview of mean curvature flow in higher codimension.

A similar contrast already appears in the Bernstein problem for minimal submanifolds. In contrast to the classical Bernstein theorem for minimal hypersurfaces, there exist nonplanar entire minimal graphs already in dimension two, as illustrated by the holomorphic graph $\C \ni z\mapsto (z,z^2)\in\C^2$, which has unbounded slope. Even imposing bounded slope does not restore rigidity. Indeed, nonplanar smooth entire minimal graphs with bounded slope can be obtained from the Lawson--Osserman construction \cite{LawsonOsserman77} with the desingularization of Ding and Yuan \cite{DingYuan06}. This is in sharp contrast with the codimension one case, where bounded slope implies rigidity by Moser's theorem, and even this assumption can be relaxed to allow controlled unbounded growth of the slope, as shown by Ecker and Huisken \cite{EckerHuisken90}.

Another important difference concerns stability. While every minimal graph of codimension one is stable, this is no longer true in higher codimension. Lawson and Osserman \cite{LawsonOsserman77}, for example, showed the existence of an unstable minimal graph over the unit disk. Thus, even in the graphical setting, the minimal surface system in higher codimension is substantially more complicated than in codimension one.

A particularly useful simplifying assumption in higher codimension is the flatness of the normal bundle. A submanifold is said to have \emph{flat normal bundle} when the curvature of its normal connection vanishes identically. In Euclidean space, by the Ricci equation, this condition is equivalent to the commutativity of the shape operators, so that at each point they can be simultaneously diagonalized. This eliminates some of the characteristic interactions between different normal directions and allows several arguments from the hypersurface theory to be recovered in higher codimension.

In the context of the Bernstein problem, Xin \cite{Xin05} exploited this structure without assuming that the submanifold is globally graphical. Under the flat normal bundle condition, the Simons identity and the equation for the $w$-function associated with the generalized Gauss map take forms analogous to those in the hypersurface case. In particular, when $w>0$, one obtains a hypersurface-type stability inequality. Using these ingredients, Xin established Bernstein type rigidity results under suitable growth assumptions.

Smoczyk, G.~Wang and Xin \cite{SmoczykWangXin06} subsequently sharpened Xin's approach. By improving the estimates for the second fundamental form, they removed the dimensional restriction in Xin's Bernstein theorem \cite{Xin05} and obtained a higher codimensional analogue of the Ecker--Huisken theorem \cite{EckerHuisken90}. They also established Schoen--Simon--Yau type curvature estimates \cite{SchoenSimonYau75} and related rigidity results.

Independently, M.-T.\ Wang \cite{Wang04} also showed that minimal graphs with flat normal bundle are stable in the usual second variation sense, obtaining corresponding curvature estimates and Bernstein type results. Thus the flat normal bundle condition provides a natural higher codimensional setting in which important parts of the stability theory, curvature estimates, and Bernstein theory for minimal hypersurfaces continue to hold. 

The flat normal bundle condition has also been used in the study of self-similar and translating solutions to mean curvature flow. Motivated by the $F$-stability theory of Colding and Minicozzi \cite{ColdingMinicozzi12} in codimension one, Arezzo and Sun \cite{ArezzoSun13} studied $F$-stability for self-shrinkers of arbitrary codimension with flat normal bundle. For translating solitons, in a different direction, the author \cite{Kunikawa15} obtained a Moser type Bernstein theorem in arbitrary codimension with flat normal bundle. This extends the hypersurface result of Bao and Shi \cite{BaoShi14}, which assumes bounded slope, to higher codimension and refines it by allowing controlled unbounded growth of the slope.

The above results naturally lead to the question of whether the flatness of the normal bundle is preserved under mean curvature flow. Smoczyk, G.~Wang and Xin \cite{SmoczykWangXin04} studied this question in connection with extending the Ecker--Huisken theory \cite{EckerHuisken89} for entire graphical mean curvature flow to arbitrary codimension, and stated a preservation result for flat normal bundles under a boundedness assumption on the second fundamental form at each time. Their proof of preservation was based on a maximum principle argument applied to the evolution equation for the normal curvature. They then used this preservation statement to derive higher codimensional analogues of several estimates for graphical mean curvature flow in the hypersurface case.

Smoczyk, Wang and Xin \cite{SmoczykWangXin04} themselves later discovered a gap in their preservation argument due to gradient terms in the evolution equation. Thus, it remained an open question whether this was a technical limitation of the proof or whether normal flatness could actually fail. Several years later, Baker remarked that the flat normal bundle condition is not preserved under mean curvature flow \cite{BakerThesis,Baker11}, although no proof or counterexample was given in either work. Baker and Nguyen \cite{BakerNguyen17}, in their study of codimension two surfaces, derived an evolution equation for the normal curvature containing gradient terms. Their formula again makes clear that these gradient terms obstruct a direct application of the maximum principle to conclude preservation of the flat normal bundle condition. Note, however, that their argument still does not show that the flatness actually fails to be preserved.

In revisiting the evolution equation for the normal curvature in arbitrary codimension, we isolate the source term for the normal curvature arising from these gradient terms, which we denote by $\Ccal$. It is a commutator term involving the covariant derivatives of the shape operators. Even when the normal bundle is initially flat, $\Ccal$ need not vanish and can therefore generate normal curvature. This gives a tensorial criterion for non-preservation of the flat normal bundle condition. If $\Ccal$ is nonzero at some point initially, then the flatness is lost instantaneously. We then construct explicit initial submanifolds satisfying this criterion. Our examples show that the gradient terms appearing in the evolution equation derived by Baker and Nguyen \cite{BakerNguyen17} can actually generate normal curvature and lead to the loss of flatness. Consequently, the preservation statement in \cite[Theorem~1]{SmoczykWangXin04} fails in general\footnote{We emphasize that this does not affect the curvature estimates and the Bernstein type results for minimal submanifolds with flat normal bundle in \cite{SmoczykWangXin06}, where preservation of the flat normal bundle under mean curvature flow is not required.}. The examples below provide explicit counterexamples confirming Baker's earlier remark \cite{BakerThesis,Baker11}.

We first realize the above criterion in a compact example by constructing an embedded two-torus in $\R^4$ whose normal bundle is initially flat and for which $\Ccal$ is nonzero at some point.

\begin{theorem}\label{thm:cpt}
There exists an embedded two-torus in $\R^4$ whose normal bundle is initially flat but does not remain flat under mean curvature flow for sufficiently small positive times.
\end{theorem}

We next give a codimension two entire graphical example in $\R^4$ whose normal bundle is initially flat. The example has bounded slope and bounded geometry. Moreover, it is uniformly area decreasing, so that the result of Savas-Halilaj and Smoczyk \cite{SavasHalilajSmoczyk24} yields a smooth graphical mean curvature flow that exists uniquely for all time. Nevertheless, $\Ccal$ is nonzero at some point of the initial graph, and the flatness of the normal bundle is lost for all sufficiently small positive times.

\begin{theorem}\label{thm:graph}
There exists a codimension two entire graph in $\R^4$ with flat normal bundle, bounded slope, bounded geometry, and the uniformly area decreasing property. The corresponding smooth graphical mean curvature flow exists uniquely for all time and preserves all of these properties except the flatness of the normal bundle. More precisely, the normal bundle is not flat for all sufficiently small positive times.
\end{theorem}

\subsection{The parallel principal normal condition}
Another simplifying condition that has been considered in higher codimension is the \emph{parallel principal normal} condition. When the mean curvature vector is nonzero, the \emph{principal normal} is the unit normal vector in its direction. It is said to be \emph{parallel} if its covariant derivative with respect to the normal connection vanishes. This condition is automatic for hypersurfaces, but restrictive in higher codimension. It nevertheless provides a convenient setting for the study of self-similar solutions. For hypersurfaces, Huisken obtained classification results for mean convex self-shrinkers in the compact and noncompact settings \cite{Huisken90,Huisken93}. Smoczyk \cite{Smoczyk05} used the parallel principal normal condition to extend these results to arbitrary codimension.

For codimension one self-shrinkers, Colding and Minicozzi \cite{ColdingMinicozzi12} removed the bounded curvature assumption from Huisken's noncompact classification. Using their method, Li and Wei \cite{LiWei14} showed that Smoczyk's classification in the complete noncompact case remains valid under weaker assumptions. Building on these results, Andrews, Li and Wei \cite{AndrewsLiWei14} obtained classification results for $F$-stable self-shrinkers with parallel principal normal.

In a different direction, using Smoczyk's method \cite{Smoczyk05}, the author \cite{Kunikawa17} extended a classification result of Mart{\'i}n, Savas-Halilaj and Smoczyk \cite{MartinSavasHalilajSmoczyk15} for translating hypersurfaces to translating solitons with parallel principal normal in arbitrary codimension. Thus the parallel principal normal condition is a useful simplifying assumption in the study of higher codimensional self-similar solutions. However, Smoczyk \cite[Remark~1.2~(iii)]{Smoczyk05} pointed out that there was no reason to expect this condition to be preserved under a general mean curvature flow.

We confirm this expectation by considering a non-isoparametric constant mean curvature (CMC) torus in $\S^3$ whose existence follows from the work of Perdomo \cite{Perdomo10}. Viewed as a codimension two submanifold of $\R^4$, its principal normal is initially parallel. By relating its Euclidean mean curvature flow to the mean curvature flow in $\S^3$, we show that the principal normal does not remain parallel for sufficiently small positive times. In codimension two, the parallel principal normal condition implies the flatness of the normal bundle, and hence the normal bundle of this example is initially flat. Although the principal normal ceases to be parallel, the normal bundle remains flat for all times for which the flow exists.

\begin{theorem}\label{thm:ppn}
There exists an embedded two-torus in $\R^4$ whose principal normal is initially parallel and whose normal bundle is initially flat, but the principal normal does not remain parallel under mean curvature flow for sufficiently small positive times. Nevertheless, the normal bundle remains flat for all times for which the flow exists.
\end{theorem}

The paper is organized as follows. In Section~\ref{sec:preliminaries}, we recall the basic notation and evolution equations for mean curvature flow in arbitrary codimension, revisit the evolution equation for the normal curvature, and derive a criterion for the loss of normal flatness in terms of the quantity $\Ccal$. In Sections~\ref{sec:cpt_example} and~\ref{sec:graph_example}, we apply this criterion to construct a compact counterexample and an entire graphical counterexample in $\R^4$ (Theorems~\ref{thm:cpt} and~\ref{thm:graph}). In Section~\ref{sec:ppn_example}, we study the parallel principal normal condition and give an example for which this condition is lost under mean curvature flow while the normal bundle remains flat (Theorem~\ref{thm:ppn}). 

\subsection*{Acknowledgements}
The author would like to thank Knut Smoczyk and Guofang Wang for helpful conversations concerning the preservation of the normal flatness under mean curvature flow. The author first learned from Knut Smoczyk, during a visit to Leibniz University Hannover in November 2016, that Smoczyk, G.~Wang and Xin had themselves discovered a gap in the preservation argument in \cite{SmoczykWangXin04}. Later, at the 8th China--Japan Geometry Conference held in Guilin in September 2023, Guofang Wang shared their earlier realization of the gap in \cite{SmoczykWangXin04} with the author and discussed the preservation problem. These conversations provided the initial motivation for the present work. The author is also grateful to Knut Smoczyk for clarifying the historical background concerning \cite{SmoczykWangXin04}. The author would also like to thank Yuanlong Xin and Naoyuki Koike for helpful comments on the manuscript. 

The author is supported by JSPS KAKENHI Grant Number 23K03105. 

\subsection*{AI Disclosure}
ChatGPT 5.6 was used in the development of this work. In particular, it was used to suggest suitable ansatzes and to explore examples satisfying the required geometric conditions. It was also used for editorial assistance in preparing the manuscript. All calculations and arguments were independently checked by the author, who takes full responsibility for the mathematical content of the article.

\section{Preliminaries and a non-preservation criterion}\label{sec:preliminaries}
\subsection{Basic notation}
Let $F:M^m\to\R^n$ be a smooth immersion. We denote by $\langle\cdot,\cdot\rangle$ the standard Euclidean inner product on $\R^n$ and by $g$ the induced metric on $M$. Let $D$ denote the standard Euclidean connection on $\R^n$. The Levi-Civita connection of $(M,g)$ and the normal connection of the normal bundle $T^\perp M$ are denoted by $\nabla$ and $\nabla^\perp$, respectively. We identify $TM$ with $dF(TM)\subset F^*T\R^n$ via $dF$.

For tangent vector fields $X,Y$ and a normal vector field $\nu$, the second fundamental form $A$ and the shape operator $S^\nu$ are defined respectively by
\begin{align*}
  A(X,Y):=(D_XY)^\perp, \qquad S^\nu(X):=-(D_X\nu)^\top,
\end{align*}
where $(\cdot)^\top$ and $(\cdot)^\perp$ denote the tangential and normal projections with respect to $F$. They are related by
\begin{align*}
  \langle A(X,Y),\nu\rangle=\langle S^\nu(X),Y\rangle.
\end{align*}
Then the Gauss and Weingarten formulas are
\begin{align*}
  D_XY=\nabla_XY+A(X,Y), \qquad D_X\nu=-S^\nu(X)+\nabla_X^\perp\nu.
\end{align*}

We regard the shape operator as the smooth bundle map
\begin{align*}
  S:T^\perp M\to\operatorname{Sym}_g(TM), \qquad \nu\mapsto S(\nu):=S^\nu, 
\end{align*}
where $\Sym_g(TM)$ denotes the bundle of self-adjoint endomorphisms of $TM$ with respect to $g$. Its covariant derivative is defined by
\begin{align*}
  (\nabla_XS)^\nu:=\nabla_XS^\nu-S^{\nabla_X^\perp\nu},
\end{align*}
where
\begin{align*}
  (\nabla_XS^\nu)(Y)=\nabla_X\bigl(S^\nu(Y)\bigr)-S^\nu(\nabla_XY).
\end{align*}

We use the convention
\begin{align*}
R^\perp(X,Y)\nu
:=
\nabla_X^\perp\nabla_Y^\perp\nu
-
\nabla_Y^\perp\nabla_X^\perp\nu
-
\nabla_{[X,Y]}^\perp\nu
\end{align*}
for the curvature of the normal connection.

\begin{definition}
We say that an immersion $F:M^m\to \R^n$ has \emph{flat normal bundle} if
\begin{align*}
R^\perp\equiv 0.
\end{align*}
\end{definition}

The Ricci equation in Euclidean space is
\begin{align*}
\langle R^\perp(X,Y)\nu,\xi\rangle
=
\langle [S^\nu,S^\xi](X),Y\rangle.
\end{align*}
Hence $F:M^m\to \R^n$ has flat normal bundle if and only if $[S^\nu,S^\xi]=0$ for all normal vectors $\nu,\xi$. Since the shape operators are self-adjoint with respect to $g$ and commute pairwise, they can be simultaneously diagonalized at each point whenever the normal bundle is flat. Note that in codimension one, the normal bundle is automatically flat. 

\subsection{Evolution equation for the shape operator}
In this subsection, we consider a smooth mean curvature flow
$F:M^m\times[0,T)\to\R^n$,
\begin{align*}
\frac{dF}{dt}(\cdot,t)=H(\cdot,t),
\end{align*}
where $H=\tr_g A$ denotes the mean curvature vector field of the immersion
$F_t=F(\cdot,t):M^m\to\R^n$. We again identify $TM$ with
$dF_t(TM)\subset F_t^*T\R^n$ via $dF_t$.

Let $X$, $Y$ and $\nu$ be time-dependent vector fields along $F_t$, with
$X, Y$ tangent and $\nu$ normal. We define
\begin{align*}
\nabla_tX:=\left(\frac{dX}{dt}\right)^\top,
\qquad
\nabla_t^\perp\nu:=\left(\frac{d\nu}{dt}\right)^\perp,
\end{align*}
and
\begin{align*}
(\nabla_t g)(X,Y)
:=
\frac{d}{dt}g(X,Y)
-g(\nabla_tX,Y)
-g(X,\nabla_tY).
\end{align*}
By definition,
\begin{align*}
\nabla_t g=0.
\end{align*}
Together with the Levi-Civita connection, $\nabla_t$ defines the spacetime tangent connection. Similarly, together with the normal connection, $\nabla_t^\perp$ defines the spacetime normal connection. These agree with the spacetime connections used in \cite[Section~2.3]{AndrewsBaker10}. They are compatible with the induced metric $g$ and the normal metric, respectively.

For the second fundamental form $A$, we define
\begin{align*}
  (\nabla_tA)(X,Y):=\nabla_t^\perp\bigl(A(X,Y)\bigr)-A(\nabla_tX,Y)-A(X,\nabla_tY).
\end{align*}
Then this can be computed as
\begin{align}\label{eq:evol_eq_A}
\nabla_tA=\Delta A+\Pcal,
\end{align}
where $\Pcal$ is a cubic polynomial in $A$. Writing $A_{ij}=A(e_i,e_j)$ and $\Pcal_{ij}=\Pcal(e_i,e_j)$ with a local orthonormal tangent frame $\{e_i\}$, we have
\begin{align*}
\Pcal_{ij}
=
\sum_{k,l}
\bigl\{
\langle A_{ij},A_{kl}\rangle A_{kl}
+\langle A_{il},A_{kl}\rangle A_{kj}
+\langle A_{jl},A_{kl}\rangle A_{ki}
-2\langle A_{ik},A_{jl}\rangle A_{kl}
\bigr\}.
\end{align*}
See \cite[Section~3]{AndrewsBaker10} for the derivation of \eqref{eq:evol_eq_A}. Similarly, for the shape operator bundle map $S$, we define 
\begin{align*}
  (\nabla_tS)^\nu(X):=\nabla_t\bigl(S^\nu(X)\bigr)-S^\nu(\nabla_tX)-S^{\nabla_t^\perp\nu}(X). 
\end{align*}
Using \eqref{eq:evol_eq_A} together with the relations
\begin{align*}
\langle (\nabla_tA)(X,Y),\nu\rangle
&=
\langle (\nabla_tS)^\nu(X),Y\rangle,\\
\langle (\Delta A)(X,Y),\nu\rangle
&=
\langle (\Delta S)^\nu(X),Y\rangle,
\end{align*}
we obtain
\begin{align}\label{eq:evol_eq_S}
(\nabla_tS)^\nu
=
(\Delta S)^\nu+\Qcal^\nu,
\end{align}
where $\Qcal^\nu$ is the self-adjoint endomorphism corresponding to $\Pcal$, defined by
\begin{align*}
\langle \Qcal^\nu(X),Y\rangle
=
\langle \Pcal(X,Y),\nu\rangle.
\end{align*}
With respect to a local orthonormal normal frame $\{\nu_\beta\}$, this is given by
\begin{align}\label{eq:Qcal_formula}
  \Qcal^\nu&=\sum_\beta\left\{
  \tr_g(S^\nu S^{\nu_\beta})S^{\nu_\beta} 
  +S^\nu(S^{\nu_\beta})^2+(S^{\nu_\beta})^2S^\nu -2S^{\nu_\beta}S^\nu S^{\nu_\beta} 
  \right\}\\
  &=\sum_\beta \left\{\tr_g(S^\nu S^{\nu_\beta})S^{\nu_\beta} +[S^{\nu_\beta}, [S^{\nu_\beta}, S^\nu]]\right\}. \notag 
\end{align}

\subsection{Evolution equations for the normal curvature}
In this subsection, we derive the evolution equation for the normal curvature tensor
$R^\perp$ under mean curvature flow. Since the following computations are tensorial,
we may compute at a point $(p,t_0)\in M\times[0,T)$ using local orthonormal tangent
and normal frames $\{e_i\}$ and $\{\nu_\alpha\}$. By parallel transport with respect to the metric connections described above, these frames may be chosen so that
\begin{align*}
\nabla e_i=0,
\qquad
\nabla_t e_i=0,
\qquad
\nabla^\perp\nu_\alpha=0,
\qquad
\nabla_t^\perp\nu_\alpha=0
\end{align*}
at $(p,t_0)$.

Differentiating the Ricci equation in time, we obtain
\begin{align*}
  \langle (\nabla_tR^\perp)(e_i,e_j)\nu_\alpha,\nu_\beta\rangle
  ={}&
  \big\langle [(\nabla_tS)^{\nu_\alpha},S^{\nu_\beta}](e_i),e_j\big\rangle
  +\big\langle [S^{\nu_\alpha},(\nabla_tS)^{\nu_\beta}](e_i),e_j\big\rangle\\[3pt]
  ={}&
  \big\langle [(\Delta S)^{\nu_\alpha},S^{\nu_\beta}](e_i),e_j\big\rangle
  +\big\langle [\Qcal^{\nu_\alpha},S^{\nu_\beta}](e_i),e_j\big\rangle\\
  &+\big\langle [S^{\nu_\alpha},(\Delta S)^{\nu_\beta}](e_i),e_j\big\rangle
  +\big\langle [S^{\nu_\alpha},\Qcal^{\nu_\beta}](e_i),e_j\big\rangle,
\end{align*}
where \eqref{eq:evol_eq_S} is used in the second equality. On the other hand, applying the rough Laplacian to the Ricci equation and, when necessary, using the orthonormal frames chosen above, we obtain
\begin{align*}
  \langle (\Delta R^\perp)(e_i,e_j)\nu_\alpha,\nu_\beta\rangle
  ={}&
  \big\langle [(\Delta S)^{\nu_\alpha},S^{\nu_\beta}](e_i),e_j\big\rangle
  +\big\langle [S^{\nu_\alpha},(\Delta S)^{\nu_\beta}](e_i),e_j\big\rangle\\
  &+2\sum_k
  \big\langle [(\nabla_{e_k}S)^{\nu_\alpha},
  (\nabla_{e_k}S)^{\nu_\beta}](e_i),e_j\big\rangle.
\end{align*}
For normal vector fields $\nu,\xi$, we define
\begin{align}\label{eq:Ccal_definition}
\Ccal(\nu,\xi)
:=
\sum_k
\big[
(\nabla_{e_k}S)^\nu,
(\nabla_{e_k}S)^\xi
\big]
=
\tr_g
\big[
(\nabla_\bullet S)^\nu,
(\nabla_\bullet S)^\xi
\big], 
\end{align}
which is independent of the choice of $\{e_k\}$. Therefore, we obtain\footnote{The term involving $\Ccal$ corresponds to the gradient term appearing in equation~(10) of Baker and Nguyen \cite{BakerNguyen17}.}
\begin{align}\label{eq:Rperp_Ccal}
  \big\langle (\nabla_tR^\perp)(e_i,e_j)\nu_\alpha,\nu_\beta\big\rangle
  ={}&
  \big\langle (\Delta R^\perp)(e_i,e_j)\nu_\alpha,\nu_\beta\big\rangle+\big\langle [\Qcal^{\nu_\alpha},S^{\nu_\beta}](e_i),e_j\big\rangle\\
  &+\big\langle [S^{\nu_\alpha},\Qcal^{\nu_\beta}](e_i),e_j\big\rangle-2\langle \Ccal(\nu_\alpha,\nu_\beta)(e_i),e_j\rangle. 
  \notag
\end{align}
Moreover, using \eqref{eq:Qcal_formula}, we have
\begin{align*}
  [\Qcal^{\nu_\alpha},S^{\nu_\beta}]
  =
  \sum_\gamma\left\{
  \tr_g(S^{\nu_\alpha}S^{\nu_\gamma})[S^{\nu_\gamma},S^{\nu_\beta}]
  +
  [[S^{\nu_\gamma},[S^{\nu_\gamma},S^{\nu_\alpha}]],S^{\nu_\beta}]
  \right\}.
\end{align*}
By the Ricci equation, this can be schematically written as
\begin{align}\label{eq:Qcal_commutator}
  [\Qcal^{\nu_\alpha},S^{\nu_\beta}]
  =
  A\ast A\ast R^\perp,
\end{align}
where we use Hamilton's $\ast$-notation. The same argument with $\nu_\alpha$ and $\nu_\beta$ interchanged applies to $[S^{\nu_\alpha},\Qcal^{\nu_\beta}]=A\ast A\ast R^\perp$. On the other hand, we have
\begin{align*}
\Ccal(\nu_\alpha,\nu_\beta)=\nabla A\ast\nabla A.
\end{align*}
Importantly, no factor involving $R^\perp$ appears in $\Ccal$. Hence, the evolution equation for $R^\perp$ can be written as
\begin{align*}
  \nabla_tR^\perp=\Delta R^\perp+ A\ast A\ast R^\perp+\nabla A\ast\nabla A.
\end{align*}
As a consequence, we have 
\begin{align}\label{eq:evol_eq_norm_Rperp}
  \left(\frac{d}{dt}-\Delta\right)|R^\perp|^2
  =
  -2|\nabla R^\perp|^2
  +A\ast A\ast R^\perp\ast R^\perp
  +\nabla A\ast\nabla A\ast R^\perp.
\end{align}
The last term in \eqref{eq:evol_eq_norm_Rperp} is absent from the evolution equation for $|R^\perp|^2$ derived in \cite{SmoczykWangXin04}. In that paper, the preservation of the flatness of the normal bundle was intended to follow from an Ecker--Huisken type maximum principle \cite{EckerHuisken89}. Even assuming that $A$ and $\nabla A$ are uniformly bounded along the flow, however, if we set
\begin{align*}
  u:=|R^\perp|^2,
\end{align*}
the corrected evolution equation \eqref{eq:evol_eq_norm_Rperp} yields only an inequality of the form
\begin{align*}
  \left(\frac{d}{dt}-\Delta\right)u
  \leq C_1u+C_2\sqrt{u}.
\end{align*}
More recently, Savas-Halilaj and Smoczyk \cite[Proposition~1.3]{SavasHalilajSmoczyk24} established a more flexible maximum principle allowing inequalities of the form
\begin{align*}
  \left(\frac{d}{dt}-\Delta\right)u
  \leq \langle \mathbf{a},\nabla u\rangle+\Phi(u),
\end{align*}
where $\mathbf{a}$ is a uniformly bounded tangent vector field and $\Phi$ is uniformly Lipschitz continuous. In our case, one may take $\mathbf{a}=0$. However, the reaction term
\begin{align*}
  \Phi(u)=C_1u+C_2\sqrt{u}
\end{align*}
is not Lipschitz at $u=0$, and hence this maximum principle still does not imply preservation of the condition $R^\perp=0$.

\subsection{A criterion for loss of normal flatness}
We now give a simple criterion for the flatness of the normal bundle to fail to be preserved under mean curvature flow. Since $\nabla S$ is tensorial, the quantity $\Ccal$ defined in \eqref{eq:Ccal_definition} is pointwise. Thus, at a fixed point, $\{e_k\}$ may be any orthonormal basis of the tangent space, and $\nu,\xi$ may be any normal vectors:  
\begin{align*}
\Ccal(\nu,\xi)
:=
\sum_k
\big[
(\nabla_{e_k}S)^\nu,
(\nabla_{e_k}S)^\xi
\big]
=
\tr_g
\big[
(\nabla_\bullet S)^\nu,
(\nabla_\bullet S)^\xi
\big]. 
\end{align*}

\begin{theorem}\label{thm:nonpreservation_criterion}
Let $F:M^m\times[0,T)\to\R^n$ be a smooth mean curvature flow, and suppose that the normal bundle of $F_0=F(\cdot,0)$ is flat. If there exist a point $p\in M$ and normal vectors $\nu,\xi$ to $F_0$ at $p$ such that
\begin{align*}
  \Ccal(\nu,\xi)\neq0,
\end{align*}
then the flatness of the normal bundle is not preserved for all sufficiently small $t>0$.
\end{theorem}

\begin{proof}
Fix the point $p\in M$ and the normal vectors $\nu,\xi$ given in the assumption. Since $\Ccal(\nu,\xi)\neq 0$, there exist $X,Y\in T_pM$ such that
\begin{align*}
\big\langle
\Ccal(\nu,\xi)(X),Y
\big\rangle
\neq 0.
\end{align*}
By assumption, $R^\perp(\cdot, 0)\equiv 0$. Hence
\begin{align*}
\Delta R^\perp(\cdot,0)\equiv 0.
\end{align*}
Moreover, by \eqref{eq:Qcal_commutator},
\begin{align*}
[\Qcal^\nu,S^\xi]
=
[S^\nu,\Qcal^\xi]
=
0
\end{align*}
along $F_0$. Therefore, evaluating \eqref{eq:Rperp_Ccal} at $(p,0)$, we obtain
\begin{align}\label{eq:Rperp_time_derivative_nonzero}
\big\langle
(\nabla_tR^\perp)(X,Y)\nu,\xi
\big\rangle
&=
-2
\big\langle
\Ccal(\nu,\xi)(X),Y
\big\rangle
\neq 0.
\end{align}
Extend $X,Y,\nu,\xi$ smoothly in the time direction at $p$, and define
\begin{align*}
\psi(t)
:=
\big\langle
R^\perp(X,Y)\nu,\xi
\big\rangle(p,t).
\end{align*}
As $R^\perp(\cdot,0)\equiv 0$, all terms arising from the time derivatives of
$X,Y,\nu,\xi$ vanish when we differentiate $\psi$ at $t=0$. Hence, by
\eqref{eq:Rperp_time_derivative_nonzero},
\begin{align*}
\frac{d\psi}{dt}(0)
&=
\big\langle
(\nabla_tR^\perp)(X,Y)\nu,\xi
\big\rangle(p,0)
\neq 0.
\end{align*}
Since $\psi(0)=0$, it follows that $\psi(t)\neq 0$ for all sufficiently small $t>0$. Therefore,
\begin{align*}
R^\perp(p,t)\neq 0
\end{align*}
for all sufficiently small $t>0$, and the flatness of the normal bundle is not preserved.
\end{proof}

\section{A compact counterexample}\label{sec:cpt_example}
In this section, we construct an explicit compact surface of codimension two in $\R^4$ with flat normal bundle. By confirming the criterion in Theorem~\ref{thm:nonpreservation_criterion}, we show that the flatness of the normal bundle is not preserved under mean curvature flow.

\subsection{The ansatz for the initial torus with flat normal bundle}
Fix $L>0$, and consider two circles in $\C$:
\begin{align*}
  \S^1_L
  &:=
  \{e^{2\pi i s/L}\in \C\mid s\in\R\},
  \qquad
  \S^1_{2\pi}
  :=
  \{e^{i\theta}\in \C\mid \theta\in\R\}.
\end{align*}
Take a smooth unit speed closed curve
\begin{align*}
  \gamma:\S^1_L\to\S^2,
\end{align*}
where $\S^2\subset\R^3$ is the standard unit sphere. Consider also a smooth closed curve
\begin{align*}
  (\rho,z):\S^1_{2\pi}\to\R^2
\end{align*}
satisfying the profile conditions
\begin{align}\label{eq:profile_conditions}
  \rho>0,\qquad
  (\rho')^2+(z')^2=1,\qquad
  \rho'z''-z'\rho''=1.
\end{align}
Differentiating the second condition in \eqref{eq:profile_conditions}, we obtain
\begin{align}\label{eq:profile_orthogonality}
  \rho'\rho''+z'z''=0.
\end{align}
Combining \eqref{eq:profile_orthogonality} with the third condition in \eqref{eq:profile_conditions}, we further obtain
\begin{align}\label{eq:profile_derivatives}
  \rho''=-z',\qquad z''=\rho'.
\end{align}

Define a smooth map from the two-torus
\begin{align*}
  F:\S^1_L\times\S^1_{2\pi}\to\R^3\times\R=\R^4
\end{align*}
by
\begin{align}\label{eq:ansatz_cpt}
  F\left(e^{2\pi i s/L},e^{i\theta}\right)
  =
  \bigl(\rho(e^{i\theta})\gamma(e^{2\pi i s/L}),z(e^{i\theta})\bigr).
\end{align}
We use $(s,\theta)$ as local coordinates on $\S^1_L\times\S^1_{2\pi}$ and simply write
\begin{align*}
F(s,\theta)=(\rho(\theta)\gamma(s),z(\theta)).
\end{align*}
Note that $s$ is the arc-length parameter of $\gamma$. Then the tangent vectors are
\begin{align*}
F_s=(\rho \Tsf,0),\qquad
F_\theta=(\rho'\gamma,z'),
\end{align*}
where we set $\Tsf:=\gamma'$. By the second condition in \eqref{eq:profile_conditions}, the induced metric is
\begin{align*}
g=\rho^2\,ds^2+d\theta^2.
\end{align*}
Since $\rho>0$, it follows that $F$ is an immersion and
\begin{align*}
e_1:=\frac{1}{\rho}F_s=(\Tsf,0),
\qquad
e_2:=F_\theta
\end{align*}
form a local orthonormal tangent frame.

Now we set
\begin{align*}
\Nsf:=\gamma\times\Tsf.
\end{align*}
Then $\{\Tsf,\Nsf,\gamma\}$ is the positively oriented orthonormal frame of $\R^3$ along $\gamma$ (the so-called \emph{Darboux frame}). If $\kappa$ denotes the geodesic curvature of $\gamma$, the Darboux formulas can be written as
\begin{align}\label{eq:Darboux_formulas}
\Tsf'&=\kappa\Nsf-\gamma,\qquad
\Nsf'=-\kappa\Tsf,\qquad
\gamma'=\Tsf.
\end{align}

An orthonormal normal frame along $F$ is given by
\begin{align*}
\nu:=(\Nsf,0),
\qquad
\xi:=(-z'\gamma,\rho').
\end{align*}
Here the second condition in \eqref{eq:profile_conditions} is used to see that $\xi$ is a unit vector. Using \eqref{eq:Darboux_formulas}, we have
\begin{align}\label{eq:compact_second_derivatives}
F_{ss}
&=
(\rho\Tsf',0)
=
\rho(\kappa\Nsf-\gamma,0),\notag\\[3pt]
F_{s\theta}
&=
F_{\theta s}
=
(\rho'\Tsf,0),\\[3pt]
F_{\theta\theta}
&=
(\rho''\gamma,z'').\notag
\end{align}
Therefore,
\begin{align*}
A(e_1,e_1)
&=(D_{e_1}e_1)^\perp
=\frac{1}{\rho^2}(F_{ss})^\perp\\[3pt]
&=\frac{\langle F_{ss},\nu\rangle}{\rho^2}\nu
+\frac{\langle F_{ss},\xi\rangle}{\rho^2}\xi\\[3pt]
&=\frac{\kappa}{\rho}\nu+\frac{z'}{\rho}\xi.
\end{align*}
A similar computation gives
\begin{align*}
  A(e_1,e_2)=A(e_2,e_1)=0,\qquad
  A(e_2,e_2)=(\rho'z''-z'\rho'')\xi=\xi,
\end{align*}
where the third condition in \eqref{eq:profile_conditions} is used in the last equality. Thus, with respect to $\{e_1,e_2\}$, the shape operators are represented by 
\begin{align*}
  S^\nu
  &=
  \begin{pmatrix}
    \kappa/\rho & 0\\[6pt]
    0 & 0
  \end{pmatrix},
  \qquad
  S^\xi
  =
  \begin{pmatrix}
    z'/\rho & 0\\[6pt]
    0 & 1
  \end{pmatrix}.
\end{align*}
It follows that $S^\nu$ and $S^\xi$ are simultaneously diagonalized with respect to $\{e_1,e_2\}$. In particular, $[S^\nu,S^\xi]=0$, and hence, by the Ricci equation, the normal bundle of $F$ is flat.

\subsection{Verification of the non-preservation criterion}
We now assume in addition that $\gamma:\S^1_L\to\S^2$ has nonconstant geodesic curvature $\kappa$. Then there exists a point $s_0 \in \R$ such that
\begin{align*}
  \kappa'(s_0)\neq0.
\end{align*}

In order to compute the quantity $\Ccal$ in the criterion of Theorem~\ref{thm:nonpreservation_criterion}, we first compute the tangent connection. Using \eqref{eq:compact_second_derivatives}, we have 
\begin{align*}
\nabla_{e_1}e_1
&=
-\frac{\rho'}{\rho}e_2,
&
\nabla_{e_1}e_2
&=
\frac{\rho'}{\rho}e_1,\\[6pt]
\nabla_{e_2}e_1
&=
0,
&
\nabla_{e_2}e_2
&=
0,
\end{align*}
where \eqref{eq:profile_orthogonality} is used in the computation of $\nabla_{e_2}e_2$. Moreover, using \eqref{eq:Darboux_formulas} and
\eqref{eq:profile_orthogonality}, we obtain
\begin{align*}
  \nabla^\perp\nu
  =
  \nabla^\perp\xi
  =
  0
\end{align*}
for our previous choice of the normal frame $\{\nu,\xi\}$. Hence,
\begin{align*}
(\nabla_XS)^\nu=\nabla_XS^\nu,
\qquad
(\nabla_XS)^\xi=\nabla_XS^\xi,
\end{align*}
for any tangent vector field $X$. Taking this into account, by differentiating
\begin{align*}
  S^\nu(e_1)=\frac{\kappa}{\rho}e_1,
\end{align*}
we have
\begin{align*}
  (\nabla_{e_1}S)^\nu(e_1)
  &=\nabla_{e_1}(S^\nu(e_1))-S^\nu(\nabla_{e_1}e_1)\\[3pt]
  &=\nabla_{e_1}\left(\frac{\kappa}{\rho}e_1\right)
  +\frac{\rho'}{\rho}S^\nu(e_2)\\[3pt]
  &=e_1\left(\frac{\kappa}{\rho}\right)e_1+\frac{\kappa}{\rho}\nabla_{e_1}e_1\\[3pt]
  &=\frac{\kappa'}{\rho^2}e_1-\frac{\kappa\rho'}{\rho^2}e_2.
\end{align*}
By a similar computation, we obtain
\begin{align*}
  (\nabla_{e_1}S)^\nu
  &=
  \frac{1}{\rho^2}
  \begin{pmatrix}
    \kappa' & -\kappa\rho'\\[6pt]
    -\kappa\rho' & 0
  \end{pmatrix},
  \qquad
  (\nabla_{e_2}S)^\nu
  =
  \frac{1}{\rho^2}
  \begin{pmatrix}
    -\kappa\rho' & 0\\[6pt]
    0 & 0
  \end{pmatrix}.
\end{align*}
For the $\xi$ direction, using \eqref{eq:profile_derivatives}, we likewise obtain
\begin{align*}
  (\nabla_{e_1}S)^\xi
  &=
  \frac{1}{\rho^2}
  \begin{pmatrix}
    0 & \rho'(\rho-z')\\[6pt]
    \rho'(\rho-z') & 0
  \end{pmatrix},
  \qquad
  (\nabla_{e_2}S)^\xi
  =
  \frac{1}{\rho^2}
  \begin{pmatrix}
    \rho'(\rho-z') & 0\\[6pt]
    0 & 0
  \end{pmatrix}.
\end{align*}
Therefore,
\begin{align*}
  \Ccal(\nu,\xi)=\sum_{i=1}^2
  \big[(\nabla_{e_i}S)^\nu,(\nabla_{e_i}S)^\xi\big]
  =
  \big[(\nabla_{e_1}S)^\nu,(\nabla_{e_1}S)^\xi\big],
\end{align*}
since $(\nabla_{e_2}S)^\nu$ and $(\nabla_{e_2}S)^\xi$ are both diagonal. Hence,
\begin{align}\label{eq:Ccal_cpt}
  \Ccal(\nu,\xi)
  =
  \frac{\kappa'\rho'(\rho-z')}{\rho^4}
  \begin{pmatrix}
    0 & 1\\
    -1 & 0
  \end{pmatrix}.
\end{align}

We now choose the profile curve explicitly by
\begin{align*}
  \rho(\theta)=2+\cos\theta,\qquad
  z(\theta)=\sin\theta.
\end{align*}
This choice satisfies all the profile conditions \eqref{eq:profile_conditions}. Moreover,
\begin{align*}
  \rho'(\rho-z')
  =
  -2\sin\theta.
\end{align*}
Thus \eqref{eq:Ccal_cpt} becomes
\begin{align*}
  \Ccal(\nu,\xi)
  =
  \frac{2\kappa'\sin\theta}{\rho^4}
  \begin{pmatrix}
    0 & -1\\
    1 & 0
  \end{pmatrix}.
\end{align*}
In particular, we have 
\begin{align*}
  \Ccal(\nu,\xi)\big|_{(s_0,\pi/2)}
  =
  \frac{\kappa'(s_0)}{8}
  \begin{pmatrix}
    0 & -1\\
    1 & 0
  \end{pmatrix}
  \neq0.
\end{align*}
Therefore, Theorem~\ref{thm:nonpreservation_criterion} shows that the flatness of the normal bundle is not preserved under the mean curvature flow for all sufficiently small positive times.

\subsection{Embeddedness and fullness of the torus}
For the explicit choice
\begin{align*}
  \rho(\theta)=2+\cos\theta,\qquad
  z(\theta)=\sin\theta,
\end{align*}
the profile curve $(\rho,z)$ is an embedded circle. Thus, if $\gamma:\S^1_L\to\S^2$ is also embedded, then the ansatz~\eqref{eq:ansatz_cpt} defines an embedded torus in $\R^4$.

Moreover, the resulting torus is full in $\R^4$ because the geodesic curvature of $\gamma$ is nonconstant. To see this, suppose, to the contrary, that the image of $F$ is contained in an affine hyperplane. Then there exist $(a,b)\in\R^3\times\R$ with $(a,b)\neq(0,0)$ and $c\in\R$ such that
\begin{align*}
  \big\langle a,(2+\cos\theta)\gamma(e^{2\pi i s/L})\big\rangle
  +b\sin\theta=c
\end{align*}
for all $s,\theta$. From this identity, we see that $\big\langle a,\gamma\big(e^{2\pi i s/L}\big)\big\rangle$ is constant in $s$. We claim that $a\neq0$. Indeed, if $a=0$, then $b\sin\theta=c$ for all $\theta$, which implies $b=c=0$, contradicting $(a,b)\neq(0,0)$. Therefore $a\neq0$, and $\gamma$ is contained in
\begin{align*}
  \S^2\cap\{x\in\R^3:\langle a,x\rangle=d\}
\end{align*}
for some constant $d$. Since this intersection contains the smooth closed curve $\gamma$, it is a circle in $\S^2$. Hence $\gamma$ has constant geodesic curvature, contradicting our choice of $\gamma$. Thus $F$ is full in $\R^4$.

In summary, we obtain the following explicit compact counterexample.
\begin{theorem}
Let $L>0$, and let
\begin{align*}
  \gamma:\S^1_L\to\S^2\subset\R^3,
  \qquad
  \S^1_L=\{e^{2\pi i s/L}\in\C\mid s\in\R\},
\end{align*}
be a smooth unit speed embedded closed curve with nonconstant geodesic curvature $\kappa$. Define
\begin{align*}
  F:\S^1_L\times\S^1_{2\pi}\to\R^3\times\R=\R^4,
  \qquad
  \S^1_{2\pi}=\{e^{i\theta}\in\C\mid\theta\in\R\},
\end{align*}
by
\begin{align*}
  F\left(e^{2\pi i s/L},e^{i\theta}\right)
  =
  \bigl((2+\cos\theta)\gamma(e^{2\pi i s/L}),\sin\theta\bigr).
\end{align*}
Then $F$ is a smooth embedded full two-torus in $\R^4$ with flat normal bundle. Moreover, if the mean curvature flow starts from $F$, then the normal bundle does not remain flat for all sufficiently small positive times.
\end{theorem}
This gives the compact counterexample stated in Theorem~\ref{thm:cpt}.

\section{An entire graphical counterexample}\label{sec:graph_example}
In this section, we construct a codimension two entire graph in $\R^4$ whose normal bundle is initially flat but does not remain flat under mean curvature flow. The example is uniformly area decreasing, with bounded slope and bounded geometry. This allows us to apply the long-time existence theorem of Savas-Halilaj and Smoczyk \cite[Theorem~5.3]{SavasHalilajSmoczyk24}.

\subsection{The ansatz for the initial entire graph with flat normal bundle}

Let $r:\R\to\R$ be a smooth function, to be determined later, and define a smooth map 
\begin{align*}
  f:\R^2\to\R^2,
  \qquad
  f(x,y)=(r(y)\cos x,r(y)\sin x).
\end{align*} 
We consider its graph
\begin{align*}
F(x,y)
=
(x,y,f(x,y))
=
\bigl(x,y,r(y)\cos x,r(y)\sin x\bigr).
\end{align*}
The tangent vectors of $F$ are
\begin{align*}
  F_x=
  (1,0,-r\sin x,r\cos x), \qquad 
  F_y=
  (0,1,r'\cos x,r'\sin x), 
\end{align*}
and the induced metric is given by 
\begin{align*}
  g=Edx^2+Gdy^2, 
\end{align*}
where 
\begin{align*}
  E:=\langle F_x, F_x \rangle=1+r^2, \qquad G:=\langle F_y, F_y\rangle=1+(r')^2. 
\end{align*}
Thus
\begin{align*}
  e_1:=\frac{1}{\sqrt E}F_x,
  \qquad
  e_2:=\frac{1}{\sqrt G}F_y
\end{align*}
form an orthonormal tangent frame.

We choose the orthonormal normal frame
\begin{align*}
  \nu
  :=
  \frac{1}{\sqrt E}
  (-r,0,-\sin x,\cos x), \qquad 
  \xi
  :=
  \frac{1}{\sqrt G}
  (0,r',-\cos x,-\sin x).
\end{align*}
The second derivatives of $F$ are
\begin{align}
  F_{xx}
  &=
  (0,0,-r\cos x,-r\sin x),\notag\\
  F_{xy}
  &=
  (0,0,-r'\sin x,r'\cos x),
  \label{eq:graph_second_derivatives}\\
  F_{yy}
  &=
  (0,0,r''\cos x,r''\sin x).
  \notag
\end{align}
Therefore,
\begin{align*}
  A(e_1,e_1)
  &=
  \frac{1}{E}(F_{xx})^\perp
  =
  \frac{r}{E\sqrt G}\xi,\\[3pt]
  A(e_1,e_2)&=A(e_2, e_1)
  =
  \frac{1}{\sqrt{EG}}(F_{xy})^\perp
  =
  \frac{r'}{E\sqrt G}\nu,\\[3pt]
  A(e_2,e_2)
  &=
  \frac{1}{G}(F_{yy})^\perp
  =
  -\frac{r''}{G\sqrt{G}}\xi.
\end{align*}
Hence, with respect to $\{e_1,e_2\}$,
\begin{align}\label{eq:graph_shape_operators}
  S^\nu
  &=
  \frac{r'}{E\sqrt G}
  \begin{pmatrix}
    0 & 1\\
    1 & 0
  \end{pmatrix},
  \qquad
  S^\xi
  =
  \frac{1}{EG\sqrt{G}}
  \begin{pmatrix}
    rG & 0\\
    0 & -Er''
  \end{pmatrix}.
\end{align}
Therefore,
\begin{align*}
  [S^\nu,S^\xi]=
  \frac{r'(Er''+rG)}{E^2G^2}
  \begin{pmatrix}
    0&-1\\
    1&0
  \end{pmatrix}.
\end{align*}
Thus, we require $r$ to satisfy
\begin{align}\label{eq:graph_flatness_condition}
Er''+rG=0, 
\end{align}
or equivalently,
\begin{align}\label{eq:graph_flat_ode}
(1+r^2)r''+r\bigl(1+(r')^2\bigr)=0.
\end{align}
Then $[S^\nu,S^\xi]=0$, and hence, by the Ricci equation, the normal bundle of $F$ is flat.

Now we discuss solutions of ODE \eqref{eq:graph_flat_ode}. Setting
$q=r'$, \eqref{eq:graph_flat_ode} is equivalent to the first-order
system
\begin{align}\label{eq:graph_first_order_system}
r'=q,
\qquad
q'=-\frac{r(1+q^2)}{1+r^2}.
\end{align}
The corresponding vector field
\begin{align*}
V(r,q)
=
\left(
q,
-\frac{r(1+q^2)}{1+r^2}
\right)
\end{align*}
is smooth on the $(r,q)$-plane. Hence, for any initial data,
\eqref{eq:graph_first_order_system} admits a unique local solution.
Along any such solution, by \eqref{eq:graph_flatness_condition}, we observe that
\begin{align*}
(EG)'
=
2r'(Er''+rG)
=
0.
\end{align*}
Therefore,
\begin{align}\label{eq:energy_curve}
EG=(1+r^2)(1+q^2)=K
\end{align}
for some constant $K$ determined by the initial data. In particular,
$K\geq 1$. The case $K=1$ gives the constant solution $r=q=0$.

We now fix $K>1$ and choose the initial data
\begin{align*}
r(0)=\sqrt{K-1},
\qquad
q(0)=0.
\end{align*}
By \eqref{eq:energy_curve},
\begin{align*}
|r|\leq\sqrt{K-1},
\qquad
|q|\leq\sqrt{K-1}.
\end{align*}
Hence the solution $(r(y),q(y))$ remains bounded and therefore extends to all $y\in\R$. The corresponding level curve in the $(r,q)$-plane is
\begin{align*}
\Gamma_K
=
\left\{
(r,q)\in\R^2
\mid
(1+r^2)(1+q^2)=K
\right\},
\end{align*}
and it admits the parametrization
\begin{align*}
\theta
\mapsto
\left(
k\cos\theta,
-\frac{k\sin\theta}{\sqrt{1+k^2\cos^2\theta}}
\right),
\qquad
\theta\in\R,
\end{align*}
where $k=\sqrt{K-1}$. Thus $\Gamma_K$ is a smooth closed curve. The vector field in
\eqref{eq:graph_first_order_system} does not vanish on $\Gamma_K$,
since $r'=q'=0$ would imply $r=q=0$, which is incompatible with
\eqref{eq:energy_curve} for $K>1$. Thus the solution through
$(\sqrt{K-1},0)$ moves periodically along $\Gamma_K$. In particular,
$r$ is a nonconstant smooth periodic function on $\R$.

We have therefore shown that, for every $K>1$, there exists a nonconstant smooth periodic function
$r=r_K:\R\to\R$ satisfying ODE \eqref{eq:graph_flat_ode}. The corresponding map
\begin{align*}
f(x,y)=(r(y)\cos x,r(y)\sin x)
\end{align*}
defines an entire graph $F(x,y)=(x,y,f(x,y))$ in $\R^4$ with flat normal bundle.

\subsection{Verification of the non-preservation criterion}
We next verify that the entire graph constructed above satisfies the
non-preservation criterion in Theorem~\ref{thm:nonpreservation_criterion}.
Recall from \eqref{eq:graph_shape_operators} that, with respect to
$\{e_1,e_2\}$,
\begin{align*}
S^\nu
&=
\frac{r'}{E\sqrt G}
\begin{pmatrix}
0&1\\
1&0
\end{pmatrix},
\qquad
S^\xi
=
\frac{r}{E\sqrt G}
\begin{pmatrix}
1&0\\
0&1
\end{pmatrix}.
\end{align*}
Indeed, \eqref{eq:graph_flatness_condition} implies
\begin{align*}
-\frac{r''}{G\sqrt G}
=
\frac{r}{E\sqrt G}.
\end{align*}
Applying this to the $(2,2)$-component of $S^\xi$ gives the second identity.

Using \eqref{eq:graph_second_derivatives}, a direct computation shows that the tangent connection is given by
\begin{align*}
  \nabla_{e_1}e_1
  &=
  -\frac{rr'}{E\sqrt G}e_2,
  &
  \nabla_{e_1}e_2
  &=
  \frac{rr'}{E\sqrt G}e_1,\\[6pt]
  \nabla_{e_2}e_1
  &=
  0,
  &
  \nabla_{e_2}e_2
  &=
  0,
\end{align*}
whereas the normal connection satisfies
\begin{align*}
  \nabla^\perp_{e_1}\nu
  &=
  \frac{1}{E\sqrt G}\xi,
  &
  \nabla^\perp_{e_1}\xi
  &=
  -\frac{1}{E\sqrt G}\nu,\\[6pt]
  \nabla^\perp_{e_2}\nu
  &=
  0,
  &
  \nabla^\perp_{e_2}\xi
  &=
  0.
\end{align*}
Remark that, unlike in the compact example in Section~\ref{sec:cpt_example}, the chosen normal frame
$\{\nu,\xi\}$ is not parallel. 

We first compute $(\nabla_{e_1}S)^\nu(e_1)$. By definition,
\begin{align*}
  (\nabla_{e_1}S)^\nu(e_1)
  &=
  \nabla_{e_1}\bigl(S^\nu(e_1)\bigr)
  -
  S^\nu\bigl(\nabla_{e_1}e_1\bigr)
  -
  S^{\nabla^\perp_{e_1}\nu}(e_1).
\end{align*}
Since $e_1=\dfrac{1}{\sqrt E}F_x$, while $r$, $E$, and $G$ depend only on $y$, using the above formulas for the tangent and normal connections together with \eqref{eq:graph_shape_operators}, we have 
\begin{align*}
  \nabla_{e_1}\bigl(S^\nu(e_1)\bigr)
  &=
  \frac{r'}{E\sqrt G}\nabla_{e_1}e_2
  =
  \frac{r(r')^2}{E^2G}e_1,\\[3pt]
  -S^\nu\bigl(\nabla_{e_1}e_1\bigr)
  &=
  \frac{rr'}{E\sqrt G}S^\nu(e_2)
  =
  \frac{r(r')^2}{E^2G}e_1,\\[3pt]
  -S^{\nabla^\perp_{e_1}\nu}(e_1)
  &=
  -\frac{1}{E\sqrt G}S^\xi(e_1)
  =
  -\frac{r}{E^2G}e_1.
\end{align*}
Hence, 
\begin{align*}
  (\nabla_{e_1}S)^\nu(e_1)
  =
  \frac{r\bigl(2(r')^2-1\bigr)}{E^2G}e_1.
\end{align*}
A similar computation gives the remaining components. With respect to $\{e_1,e_2\}$, we have
\begin{align*}
  (\nabla_{e_1}S)^\nu
  &=
  \frac{r}{E^2G}
  \begin{pmatrix}
    2(r')^2-1 & 0\\[4pt]
    0 & -2(r')^2-1
  \end{pmatrix},
  \qquad
  (\nabla_{e_1}S)^\xi
  =
  \frac{r'}{E^2G}
  \begin{pmatrix}
    0 & 1\\
    1 & 0
  \end{pmatrix}.
\end{align*}
On the other hand, since
\begin{align*}
  \nabla_{e_2}e_1
=
\nabla_{e_2}e_2
=
0,
\qquad
\nabla^\perp_{e_2}\nu
=
\nabla^\perp_{e_2}\xi
=
0,
\end{align*}
we see that $(\nabla_{e_2}S)^\xi$ is a scalar multiple of the $2\times 2$ identity matrix. Hence
\begin{align*}
  \big[
    (\nabla_{e_2}S)^\nu,
    (\nabla_{e_2}S)^\xi
  \big]
  =
  0. 
\end{align*}
It follows that
\begin{align*}
  \Ccal(\nu,\xi)
  =\big[
    (\nabla_{e_1}S)^\nu,
    (\nabla_{e_1}S)^\xi
  \big]
  =
  \frac{4r(r')^3}{E^4G^2}
  \begin{pmatrix}
    0&1\\
    -1&0
  \end{pmatrix}.
\end{align*}
Since $r$ is nonconstant, there exists a point $y_0\in\R$ such that
\begin{align*}
  r(y_0)r'(y_0)\neq0.
\end{align*}
Hence
\begin{align*}
  \Ccal(\nu,\xi)\big|_{(x,y_0)}
  \neq0.
\end{align*}
By Theorem~\ref{thm:nonpreservation_criterion}, as long as the mean curvature
flow exists, the flatness of the normal bundle is not preserved for all
sufficiently small positive times. 

\subsection{Bounded slope, bounded geometry, and area-decreasing property}
We next verify that the entire graph $F(x,y)=(x,y,f(x,y))$ constructed
above has bounded slope and bounded geometry. We also show that it is
uniformly area decreasing.

We first verify that the graph has bounded slope. Since the differential of the map
$f:\R^2\to\R^2$ is
\begin{align*}
df=
\begin{pmatrix}
-r\sin x & r'\cos x\\
r\cos x & r'\sin x
\end{pmatrix},
\end{align*}
the slope function of the graph is
\begin{align*}
v
=
\sqrt{\det\bigl(I+{}^t(df)df\bigr)}
=
\sqrt{(1+r^2)\bigl(1+(r')^2\bigr)}
=
\sqrt{EG}
=
\sqrt{K},
\end{align*}
where $I$ denotes the $2\times2$ identity matrix and we used
\eqref{eq:energy_curve} in the last equality. Thus the graph has bounded slope
(actually, it is constant\footnote{The boundedness of $|df|$ also follows directly from
$|r|,|r'|\leq \sqrt{K-1}$.}).

We next verify the bounded geometry. Since $r$ is smooth and periodic, all derivatives of $r$ are bounded, and so are all derivatives of $F$ and of the coefficients of $g$. Moreover, since $E,G\geq1$, the coefficients of $g^{-1}$ and all of their derivatives are also bounded. Therefore the components of $A$ and all of their covariant derivatives are bounded. It follows that, for every $\ell\geq0$,
\begin{align*}
\sup_{\R^2}|\nabla^\ell A|
<
\infty.
\end{align*}

Finally, we verify that the graph is uniformly area decreasing. Since the singular values of $df$ are $|r|$ and $|r'|$, the function $\psf$
introduced in \cite[Section~4]{SavasHalilajSmoczyk24} is given in our case by
\begin{align*}
\psf
=
\frac{1-r^2(r')^2}
     {(1+r^2)(1+(r')^2)}
=
\frac{1-r^2(r')^2}{K}.
\end{align*}
By \eqref{eq:energy_curve} and the arithmetic-geometric mean inequality, 
\begin{align*}
  r^2(r')^2
  &=
  K-(1+r^2)-\bigl(1+(r')^2\bigr)+1\\
  &\leq
  K-2\sqrt K+1
  =
  (\sqrt K-1)^2.
\end{align*}
Hence, if we choose $1<K<4$, then
\begin{align*}
  \psf
  \geq
  \frac{1-(\sqrt K-1)^2}{K}
  =:\varepsilon
  >
  0.
\end{align*}
Thus $\psf$ has a positive lower bound $\varepsilon$, and hence the graph is uniformly area decreasing in the sense of \cite[Section~4]{SavasHalilajSmoczyk24}.

In summary, we obtain the following theorem, which provides the initial submanifold used in Theorem~\ref{thm:graph}.
\begin{theorem}\label{thm:entire_graph_initial}
For every $1<K<4$, let $r=r_K:\R\to\R$ be the solution of
\begin{align*}
  (1+r^2)r''+r\bigl(1+(r')^2\bigr)=0
\end{align*}
with initial data
\begin{align*}
  r(0)=\sqrt{K-1},
  \qquad
  r'(0)=0.
\end{align*}
Then $r$ is a nonconstant smooth periodic function. The corresponding
codimension two entire graph 
\begin{align*}
  F(x,y)
  =
  \bigl(x,y,r(y)\cos x,r(y)\sin x\bigr)
\end{align*}
 in $\R^4$ has flat normal bundle, bounded slope, and bounded geometry. Moreover, it is
uniformly area decreasing. In addition, the quantity $\Ccal$ is nonzero at
some point.
\end{theorem}

\subsection{Mean curvature flow starting from the entire graph}
We now consider the mean curvature flow starting from the entire graph
constructed in Theorem~\ref{thm:entire_graph_initial}. Since the initial
submanifold is noncompact, short-time existence and uniqueness of the mean
curvature flow are in general not automatic, and we therefore appeal to
the results of Savas-Halilaj and Smoczyk. By
\cite[Proposition~2.1]{SavasHalilajSmoczyk24}, the bounded slope and bounded
geometry established in the previous subsection imply that there exists a
unique smooth entire graphical mean curvature flow on some time interval
$[0,T)$, $T>0$, in the class considered there.

We denote the evolving graph at time $t$ by $M_t$. Since the initial graph
$M_0$ in Theorem~\ref{thm:entire_graph_initial} has nonvanishing $\Ccal$,
Theorem~\ref{thm:nonpreservation_criterion} implies that the flatness of
the normal bundle is lost immediately along the flow. More precisely, the
normal bundle of $M_t$ is not flat for all sufficiently small $t>0$.

Finally, in addition to the bounded geometry of $M_0$, we have shown that $M_0$ is uniformly area decreasing with $\psf\geq\varepsilon>0$. Hence \cite[Theorem~5.3]{SavasHalilajSmoczyk24} applies and yields a unique graphical mean curvature flow defined for all $t\geq0$ satisfying 
\begin{align*}
\sup_{M_t}|\nabla^\ell A|\leq C(\ell)<\infty
\end{align*}
for every $\ell\geq0$ and $t\geq0$. Moreover, the lower bound $\psf\geq\varepsilon$ is preserved along the flow. It follows that $|df|$, and hence the slope function $v$, remains bounded for all $t\geq0$. This completes the proof of Theorem~\ref{thm:graph}.

\section{Loss of the parallel principal normal condition}\label{sec:ppn_example}
\subsection{Parallel principal normal and flat normal bundle}
Let $F:M^m\to\R^n$ be a smooth immersion with nonvanishing mean curvature vector field $H$. The unit normal vector field
\begin{align*}
\n=\frac{H}{|H|}
\end{align*}
is called the \emph{principal normal}.

\begin{definition}
We say that $F:M\to \R^n$ has \emph{parallel principal normal} if
\begin{align*}
\nabla^\perp\n=0.
\end{align*}
\end{definition}

This condition is automatic for hypersurfaces, whereas in higher codimension it imposes a nontrivial restriction on the normal connection.

Suppose now that the codimension $n-m$ is two. If the principal normal $\n$ is parallel, then, by the definition of the normal curvature,
\begin{align*}
R^\perp(X,Y)\n=0
\end{align*}
for all tangent vector fields $X,Y$. Let $\b$ be a local unit normal field orthogonal to $\n$. Since $R^\perp(X,Y)$ is skew-symmetric on the normal space, we have
\begin{align*}
\langle R^\perp(X,Y)\b,\n\rangle
&=
-\langle \b,R^\perp(X,Y)\n\rangle
=
0,\\
\langle R^\perp(X,Y)\b,\b\rangle
&=
0.
\end{align*}
Hence $R^\perp(X,Y)\b=0$, and therefore
\begin{align*}
R^\perp(X,Y)=0.
\end{align*}
Thus, in codimension two, the parallel principal normal condition implies the flatness of the normal bundle, as noted for example by Smoczyk \cite[Remark~1.2~(ii)]{Smoczyk05}. In higher codimension, however, the parallel principal normal condition does not in general imply the flatness of the normal bundle.

\subsection{Non-isoparametric CMC surfaces in the three-sphere}
Let $\phi:M^2\to\S^3$ be a two-sided immersion, where $\S^3$ is the unit three-sphere in $\R^4$, and let $\nu$ be a choice of global unit normal vector field along $\phi$. Let $B$ be the second fundamental form of $\phi$. Then the scalar mean curvature of $\phi$ with respect to $\nu$ is globally defined by
\begin{align*}
h
=
\tr_g\langle B(\cdot, \cdot),\nu\rangle.
\end{align*}

\begin{lemma}\label{lem:nonisoparametric_B}
Let $\phi:M^2\to\S^3$ be a two-sided\footnote{The two-sidedness assumption is imposed only to simplify the argument. The same conclusion holds without it by working locally with a unit normal vector field.} non-isoparametric CMC immersion. Then $|B|^2$ is nonconstant on $M$.
\end{lemma}

\begin{proof}
Suppose, to the contrary, that $|B|^2$ is constant. At each point, let $\kappa_1$ and $\kappa_2$ denote the principal curvatures of $M$ in $\S^3$. Then
\begin{align*}
h
&=
\kappa_1+\kappa_2,
&
|B|^2
&=
\kappa_1^2+\kappa_2^2.
\end{align*}
Since $h$ and $|B|^2$ are constant,
\begin{align*}
\kappa_1\kappa_2
=
\frac{h^2-|B|^2}{2}
\end{align*}
is also constant. It follows that $\kappa_1$ and $\kappa_2$ are the two roots of
\begin{align*}
\lambda^2-h\lambda+\frac{h^2-|B|^2}{2}=0.
\end{align*}
Since the coefficients are constant, the unordered pair $\{\kappa_1,\kappa_2\}$ is independent of the point, so $\phi$ is isoparametric, a contradiction.
\end{proof}

\subsection{Spherical mean curvature flow and rescaling}
Let $\phi:M^2\to\S^3$ be a compact non-isoparametric CMC immersion with nonzero mean curvature. Then its mean curvature vector determines a global unit normal vector field $\nu$ along $\phi$. Let
\begin{align*}
\xi:M\times[0,\tau_*)\to\S^3
\end{align*}
be the mean curvature flow in the unit three-sphere starting from $\phi$, that is,
\begin{align*}
\frac{d\xi}{d\tau}(\cdot,\tau)
=
h(\cdot,\tau)\nu(\cdot,\tau),
\qquad
\xi(\cdot,0)
=
\phi, 
\end{align*}
where $\nu(\cdot,\tau)$ is a smooth global unit normal vector field along $\xi_\tau=\xi(\cdot, \tau)$ extending the initial choice. We refer to this flow as the \emph{spherical mean curvature flow} starting from $\phi$. Here and below, along $\xi_\tau$, we denote by $B$ the second fundamental form in $\S^3$ and by $h$ the scalar mean curvature with respect to $\nu$.

By \cite{Huisken87}, the scalar mean curvature satisfies
\begin{align}\label{eq:spherical_mean_curvature_evolution}
\frac{dh}{d\tau}
=
\Delta h+\bigl(|B|^2+2\bigr)h.
\end{align}
Using this formula, we obtain the following.

\begin{lemma}\label{lem:spherical_h_nonconstant}
For all sufficiently small $\tau>0$, $h(\cdot,\tau)$ is nonconstant on $M$.
\end{lemma}

\begin{proof}
Since $h$ is a nonzero constant at $\tau=0$, we have by \eqref{eq:spherical_mean_curvature_evolution},
\begin{align}\label{eq:initial_mean_curvature_derivative}
\frac{dh}{d\tau}(\cdot, 0)
=
h\bigl(|B|^2+2\bigr).
\end{align}
On the other hand, $|B|^2$ is nonconstant at $\tau=0$ by Lemma~\ref{lem:nonisoparametric_B}. Hence the right-hand side of \eqref{eq:initial_mean_curvature_derivative} is nonconstant on $M$. Therefore there exist $p_1,p_2\in M$ such that
\begin{align*}
\frac{dh}{d\tau}(p_1,0)
\neq
\frac{dh}{d\tau}(p_2,0). 
\end{align*}
This means that $h(p_1,\tau)-h(p_2,\tau)$ is strictly increasing or decreasing for all sufficiently small $\tau>0$. Since 
\begin{align*}
h(p_1,0)-h(p_2,0)=0, 
\end{align*}
it follows that
\begin{align*}
h(p_1,\tau)-h(p_2,\tau)\neq 0
\end{align*}
for all sufficiently small $\tau>0$. Thus $h(\cdot,\tau)$ is nonconstant. 
\end{proof}

We next rescale the spherical mean curvature flow to obtain a Euclidean mean curvature flow in $\R^4$. Define $F:M\times[0,T)\to\R^4$ by
\begin{align}\label{eq:ppn_scaling}
F(p,t)
&=
r(t)\xi(p,\tau(t)),
&
r(t)
&=
\sqrt{1-4t},
&
\tau(t)
&=
-\frac14\log(1-4t),
\end{align}
for $t\in[0,T)$, where $T>0$ is chosen so that $\tau(t)<\tau_*$. This rescaling separates the Euclidean mean curvature flow into the shrinking radial factor and the spherical mean curvature flow. 

\begin{lemma}\label{lem:spherical_euclidean_mcf}
The map $F:M\times[0,T)\to\R^4$ defined by \eqref{eq:ppn_scaling} is a mean curvature flow in $\R^4$. Moreover, for every $t\in[0,T)$, the time-slice $F_t=F(\cdot,t)$ is an immersed hypersurface of $\S^3(r(t))$.
\end{lemma}

\begin{proof}
Viewing $\xi_\tau=\xi(\cdot,\tau):M^2\to\S^3\subset\R^4$ as an immersion into $\R^4$, its Euclidean mean curvature vector is $h\nu-2\xi$. Since dilation by the factor $r$ multiplies the mean curvature vector by $r^{-1}$, the Euclidean mean curvature vector $H$ of $F_t$ is
\begin{align*}
H
=
\frac{1}{r}\bigl(h\nu-2\xi\bigr).
\end{align*}

On the other hand, by \eqref{eq:ppn_scaling}, we have
\begin{align*}
\frac{dr}{dt}
=
-\frac{2}{r},
\qquad
\frac{d\tau}{dt}
=
\frac{1}{r^2}.
\end{align*}
Therefore
\begin{align*}
\frac{dF}{dt}
&=
\frac{dr}{dt}\xi
+
r\frac{d\tau}{dt}\frac{d\xi}{d\tau}\\[6pt]
&=
-\frac{2}{r}\xi
+
\frac{1}{r}h\nu=H.
\end{align*}
Thus $F$ satisfies the mean curvature flow in $\R^4$.

The second assertion of the lemma follows immediately from $|F(\cdot,t)|=r(t)$.
\end{proof}

We can now describe the mechanism leading to the loss of the parallel principal normal condition.

\begin{theorem}\label{thm:cmc_ppn}
Let $\phi:M^2\to\S^3$ be a compact non-isoparametric CMC immersion with nonzero mean curvature, and $\xi:M\times[0,\tau_*)\to\S^3$ be a spherical mean curvature flow starting from $\phi$. Consider the mean curvature flow $F:M\times[0,T)\to\R^4$ obtained from $\xi$ by the rescaling \eqref{eq:ppn_scaling}. Then 
\begin{itemize}
\item[(1)] The normal bundle of $F_t=F(\cdot, t):M\to \R^4$ is flat for every $t\in[0,T)$.
\item[(2)] The principal normal of $F_0$ is parallel, but it is not parallel along $F_t$ for all sufficiently small $t>0$.
\end{itemize}
\end{theorem}

\begin{proof}
By Lemma~\ref{lem:spherical_euclidean_mcf}, $F_t$ is an immersed hypersurface of $\S^3(r(t))$ for every $t\in[0,T)$. It is a classical fact that, for a hypersurface of a sphere, the unit normal $\nu$ in the sphere and the radial unit normal $\xi$ are parallel with respect to the Euclidean normal connection, that is,
\begin{align*}
\nabla^\perp\nu
=
\nabla^\perp\xi
=
0.
\end{align*}
In particular, the Euclidean normal bundle is flat. Thus (1) follows.

Next we show~(2). As computed in the proof of Lemma~\ref{lem:spherical_euclidean_mcf}, the Euclidean mean curvature vector of $F_t$ is
\begin{align*}
H
=
\frac{1}{r}\bigl(h\nu-2\xi\bigr), 
\end{align*}
and therefore
\begin{align*}
|H|^2
=
\frac{h^2+4}{r^2}>0.
\end{align*}
Hence the principal normal is
\begin{align*}
\n
=
\frac{H}{|H|}
=
\frac{h\nu-2\xi}{\sqrt{h^2+4}}. 
\end{align*}
Since $\nu$ and $\xi$ are parallel in the Euclidean normal bundle, we have 
\begin{align*}
\nabla_X^\perp\n
=
\frac{2X(h)}{(h^2+4)^{3/2}}
\bigl(2\nu+h\xi\bigr).
\end{align*}
Thus the principal normal is parallel if and only if $h$ is spatially constant.

At $t=0$, we have $\tau=0$, and hence $h$ is constant by assumption. Therefore the principal normal is initially parallel. On the other hand, by Lemma~\ref{lem:spherical_h_nonconstant}, $h(\cdot,\tau)$ is nonconstant for all sufficiently small $\tau>0$. It follows that the principal normal of $F_t$ is not parallel for all sufficiently small $t>0$.
\end{proof}

\subsection{An embedded torus}
We finally note that there exist embedded examples to which Theorem~\ref{thm:cmc_ppn} applies. Perdomo \cite{Perdomo10} constructed embedded non-isoparametric CMC tori in the unit three-sphere $\S^3$. Choose one such torus with nonzero mean curvature and denote the corresponding embedding by $\phi:M^2\to\S^3\subset\R^4$. By Theorem~\ref{thm:cmc_ppn}, the corresponding Euclidean mean curvature flow obtained from $\phi$ has flat normal bundle for all times for which the flow exists, while its principal normal is initially parallel but loses this property for all sufficiently small positive times. Hence Theorem~\ref{thm:ppn} follows.



\begin{thebibliography}{99}

\bibitem{AndrewsBaker10}
B.~Andrews and C.~Baker,
\emph{Mean curvature flow of pinched submanifolds to spheres},
J. Differential Geom. \textbf{85} (2010), no.~3, 357--395.

\bibitem{AndrewsLiWei14}
B.~Andrews, H.~Li and Y.~Wei,
\emph{$\mathcal{F}$-stability for self-shrinking solutions to mean curvature flow},
Asian J. Math. \textbf{18} (2014), no.~5, 757--778.

\bibitem{ArezzoSun13}
C.~Arezzo and J.~Sun,
\emph{Self-shrinkers for the mean curvature flow in arbitrary codimension},
Math. Z. \textbf{274} (2013), no.~3--4, 993--1027. 

\bibitem{BakerThesis}
C.~Baker,
\emph{The mean curvature flow of submanifolds of high codimension},
Ph.D. thesis, Australian National University, 2010,
arXiv:1104.4409.

\bibitem{Baker11}
C.~Baker,
\emph{A partial classification of type I singularities of the mean curvature flow in high codimension},
preprint, arXiv:1104.4592 (2011).

\bibitem{BakerNguyen17}
C.~Baker and H.~T.~Nguyen,
\emph{Codimension two surfaces pinched by normal curvature evolving by mean curvature flow},
Ann. Inst. H. Poincar\'e Anal. Non Lin\'eaire
\textbf{34} (2017), no.~6, 1599--1610.

\bibitem{BaoShi14}
C.~Bao and Y.~Shi,
\emph{Gauss map of translating solitons of mean curvature flow},
Proc. Amer. Math. Soc. \textbf{142} (2014), no.~12, 4333--4339.

\bibitem{ColdingMinicozzi12}
T.~H.~Colding and W.~P.~Minicozzi II,
\emph{Generic mean curvature flow I; generic singularities},
Ann. of Math. (2) \textbf{175} (2012), no.~2, 755--833.

\bibitem{DingYuan06}
W.~Ding and Y.~Yuan,
\emph{Resolving the singularities of the minimal Hopf cones},
J. Partial Differential Equations \textbf{19} (2006), no.~3, 218--231.

\bibitem{EckerHuisken89}
K.~Ecker and G.~Huisken,
\emph{Mean curvature evolution of entire graphs},
Ann. of Math. (2) \textbf{130} (1989), no.~3, 453--471.

\bibitem{EckerHuisken90}
K.~Ecker and G.~Huisken,
\emph{A Bernstein result for minimal graphs of controlled growth},
J. Differential Geom. \textbf{31} (1990), no.~2, 397--400.

\bibitem{Huisken87}
G.~Huisken,
\emph{Deforming hypersurfaces of the sphere by their mean curvature},
Math. Z. \textbf{195} (1987), no.~2, 205--219.

\bibitem{Huisken90}
G.~Huisken,
\emph{Asymptotic behavior for singularities of the mean curvature flow},
J. Differential Geom. \textbf{31} (1990), no.~1, 285--299.

\bibitem{Huisken93}
G.~Huisken,
\emph{Local and global behaviour of hypersurfaces moving by mean curvature},
in \emph{Differential Geometry: Partial Differential Equations on Manifolds}
(Los Angeles, CA, 1990),
Proc. Sympos. Pure Math. \textbf{54}, Part~1,
Amer. Math. Soc., Providence, RI, 1993, pp.~175--191.

\bibitem{Kunikawa15}
K.~Kunikawa,
\emph{Bernstein-type theorem of translating solitons in arbitrary codimension with flat normal bundle},
Calc. Var. Partial Differential Equations \textbf{54} (2015), no.~2, 1331--1344.

\bibitem{Kunikawa17}
K.~Kunikawa,
\emph{Translating solitons in arbitrary codimension},
Asian J. Math. \textbf{21} (2017), no.~5, 855--872.

\bibitem{LawsonOsserman77}
H.~B.~Lawson and R.~Osserman,
\emph{Non-existence, non-uniqueness and irregularity of solutions to the minimal surface system},
Acta Math. \textbf{139} (1977), no.~1--2, 1--17.

\bibitem{LiWei14}
H.~Li and Y.~Wei,
\emph{Classification and rigidity of self-shrinkers in the mean curvature flow},
J. Math. Soc. Japan \textbf{66} (2014), no.~3, 709--734.

\bibitem{MartinSavasHalilajSmoczyk15}
F.~Mart{\'i}n, A.~Savas-Halilaj and K.~Smoczyk,
\emph{On the topology of translating solitons of the mean curvature flow},
Calc. Var. Partial Differential Equations \textbf{54} (2015), no.~3, 2853--2882.

\bibitem{Perdomo10}
O.~M.~Perdomo,
\emph{Embedded constant mean curvature hypersurfaces on spheres},
Asian J. Math. \textbf{14} (2010), no.~1, 73--108.

\bibitem{SavasHalilajSmoczyk24}
A.~Savas-Halilaj and K.~Smoczyk,
\emph{Codimension two mean curvature flow of entire graphs},
J. London Math. Soc. \textbf{110} (2024), no.~5, e13000.

\bibitem{SchoenSimonYau75}
R.~Schoen, L.~Simon and S.-T.~Yau,
\emph{Curvature estimates for minimal hypersurfaces},
Acta Math. \textbf{134} (1975), no.~3--4, 275--288.

\bibitem{Smoczyk05}
K.~Smoczyk,
\emph{Self-shrinkers of the mean curvature flow in arbitrary codimension},
Int. Math. Res. Not. \textbf{2005} (2005), no.~48, 2983--3004.

\bibitem{Smoczyk12}
K.~Smoczyk,
\emph{Mean curvature flow in higher codimension: introduction and survey},
in \emph{Global Differential Geometry},
Springer Proc. Math. \textbf{17},
Springer, Heidelberg, 2012, pp.~231--274.

\bibitem{SmoczykWangXin04}
K.~Smoczyk, G.~Wang and Y.~L.~Xin,
\emph{Mean curvature flow with flat normal bundles},
preprint, arXiv:math/0411010v1 (2004).

\bibitem{SmoczykWangXin06}
K.~Smoczyk, G.~Wang and Y.~L.~Xin,
\emph{Bernstein type theorems with flat normal bundle},
Calc. Var. Partial Differential Equations \textbf{26} (2006), no.~1, 57--67.

\bibitem{Wang04}
M.-T.~Wang,
\emph{Stability and curvature estimates for minimal graphs with flat normal bundles},
preprint, arXiv:math/0411169v2 (2004).

\bibitem{Xin05}
Y.~L.~Xin,
\emph{Bernstein type theorems without graphic condition},
Asian J. Math. \textbf{9} (2005), no.~1, 31--44.

\end{thebibliography}
\end{document}